\documentclass{article}
\usepackage[utf8]{inputenc}
\usepackage[a4paper,total={5.8in,7.8in}]{geometry}
\usepackage{comment, amsmath, amssymb, amsthm, authblk}
\usepackage[mathlines]{lineno}

\newtheorem{lem}{Lemma}[section]
\newtheorem{cor}{Corollary}[section]
\newtheorem{pro}{Proposition}[section]
\newtheorem{teo}{Theorem}[section]
\theoremstyle{definition}
\newtheorem{defi}{Definition}[section]

\title{Geodesic flows of compact higher genus surfaces without conjugate points have expansive factors}
\author{Edhin Franklin Mamani\thanks{edy.f.cas@mat.puc-rio.br}}
\affil{Departamento de Matemática, Pontifícia Universidade Católica do Rio de Janeiro, Rua Marquês de São Vicente, 225, Rio de Janeiro 22451-900, Brazil}
\date{}

\begin{document}

\maketitle

\begin{abstract}
In this paper we show that a geodesic flow of a compact surface without conjugate points of genus greater than one is time-preserving semi-conjugate to a continuous expansive flow which is topologically mixing and has a local product structure. As an application we show that the geodesic flow of a compact surface without conjugate points of genus greater than one has a unique measure of maximal entropy. This gives an alternative proof of Climenhaga-Knieper-War Theorem. 
\end{abstract}

\section{Introduction}
The fundamental work of Morse \cite{morse24} is the main motivation for the study of existence of equivalences between geodesic flows of compact higher genus surfaces and hyperbolic geodesic flows. In 1980, Gromov \cite{grom87} built a semi-conjugacy, not necessarily time-preserving, between the geodesic flow of a higher genus surface of non-positive curvature and the geodesic flow of a hyperbolic surface. Later, Ghys did the same for Anosov geodesic flows (his proof extends to compact surfaces without conjugate points) \cite{ghys84}. In 2018, Gelfert and Ruggiero \cite{gelf19} found a time-preserving semi-conjugacy between geodesic flows of compact surfaces without focal points and expansive flows. This result was extended to the case of compact surfaces without conjugate points and continuous Green bundles \cite{gelf20}. The main contribution is the following.
\begin{teo}\label{main1}
Let $M$ be a compact surface without conjugate points of genus greater than one and $\phi_t$ be its geodesic flow. Then, there exists a continuous expansive flow $\psi_t$ that is time-preserving semi-conjugate to $\phi_t$, acting on a compact metric space $X$ of topological dimension at least two. Moreover, $\psi_t$ is topologically mixing and has a local product structure. 
\end{teo}
The good dynamical properties of the factor flow $\psi_t$ allow us to apply Bowen-Franco's classical theory to get the unique measure of maximal entropy of $\psi_t$ \cite{fran77}. Combining this fact with Buzzi-Fisher-Sambarino-Vasquez's work \cite{buzzi12} about extensions of expansive flows we get the following application.
\begin{teo}\label{umme}
Let $M$ be a compact surface without conjugate points of genus greater than one. Then, its geodesic flow $\phi_t$ has a unique measure of maximal entropy.
\end{teo}
This theorem was proved by Climenhaga, Knieper and War in 2020 for a family of compact $n$-manifolds without conjugate points, including compact higher genus surfaces \cite{clim21}. Their approach uses an extension of Bowen-Franco's theory done by Climenhaga and Thompson \cite{clim16}. Gelfert-Ruggiero's approach (assuming no focal points or continuous Green bundles) is different and gives a more direct alternative proof.

Let us briefly explain the new contributions of this work. The construction of the factor flow $\psi_t$ is based on an equivalence relation equivariant by the geodesic flow, which give rises a quotient space $X$ and a quotient flow $\psi_t$. A careful study of a basis of the quotient topology of $X$ is crucial in Gelfert-Ruggiero's works. This study relies on the structure of the expansive set $\mathcal{R}_0$, i.e., the set of points whose equivalence class is a singleton. Assuming continuity of Green bundles, $\mathcal{R}_0$ forms an open dense subset of the unit tangent bundle. However, in our setting, $\mathcal{R}_0$ might not be open because the openness of $\mathcal{R}_0$ is an open problem in general. Despite this, we show that $\mathcal{R}_0$ is dense on the complement of the set of points with non-trivial equivalence class, that turns out to be enough to get a basis of the quotient topology. 

Secondly, many ideas for the proof of the dynamical properties of $\psi_t$ in Gelfert-Ruggiero's papers rely on the fact that $X$ admits a 3-manifold structure. This allows us to choose a Riemannian metric in $X$ that can be lifted to the universal covering $\tilde{X}$ of $X$. The Riemannian structure of $\tilde{X}$ is important in the proofs. In our context, we do not know if $X$ admits a manifold structure. It is metrizable, but a metric in $X$ might not be a length metric, so it might not be lifted to $\tilde{X}$. We tackle this problem with more general topological arguments.

Finally, we study the possible loss of topological dimension of the quotient and show that topological dimension of $\tilde{X}$ is at least two.

The paper is organized as follows. Section 2 contains the preliminaries. In Section 3, we define the equivalence relation that give rises the quotient space and flow. Section 4 studies the topological properties of the factor flow. Section 5 deals with the topological dimension of the quotient space. Section 6 is concerned with the dynamical properties of the factor flow and we complete the proof of Theorem \ref{main1}. Finally, Section 7 is devoted to show the uniqueness of the measure of maximal entropy of the geodesic flow.

\section{Preliminaries}

\subsection{Compact manifolds without conjugate points}\label{m}
In this subsection we give basic definitions and notations that we use throughout the paper. Let $(M,g)$ be a $C^{\infty}$ compact connected Riemannian manifold, $TM$ be its tangent bundle and $T_1M$ be its unit tangent bundle. Consider the universal covering $\tilde{M}$ of $M$, the covering map $\pi:\tilde{M}\to M$ and the natural projection $d\pi:T\tilde{M}\to TM$. The universal covering $(\tilde{M},\tilde{g})$ is a complete Riemannian manifold with the pullback metric $\tilde{g}=\pi^*g$. A manifold $M$ has no conjugate points if the exponential map $\exp_p$ is non-singular at every $p\in M$. In particular, $\exp_p$ is a covering map for every $p\in M$ (p. 151 of \cite{doca92}).

Denote by $\nabla$ the Levi-Civita connection induced by $g$. A geodesic is a smooth curve $\gamma\subset M$ with $\nabla_{\dot{\gamma}}\dot{\gamma}=0$. For every $\theta=(p,v)\in TM$, denote by $\gamma_{\theta}$ the unique geodesic with initial conditions $\gamma_{\theta}(0)=p$ and $\dot{\gamma}_{\theta}(0)=v$. The geodesic flow $\phi_t$ is defined by 
\[ \phi: \mathbb{R}\times TM\to TM \qquad (t,\theta)\mapsto \phi_t(\theta)=\dot{\gamma}_{\theta}(t). \]
Parameterizing all geodesics by arc-length allows us to restrict $\phi_t$ to $T_1M$.

We now define a Riemannian metric on the tangent bundle $TM$ (Section 1.3 of \cite{pater97}). Denote by $P:TM\to M$ and $\tilde{P}:T\tilde{M}\to \tilde{M}$ the corresponding canonical projections. For every $\theta=(p,v)\in TM$, the Levi-Civita connection induces the so-called connection map $C_{\theta}:T_{\theta}TM\to T_pM$. These linear maps provide the linear isomorphism $T_{\theta}TM\to T_pM\oplus T_pM$ with $\xi\mapsto (d_{\theta}P(\xi),C_{\theta}(\xi))$. We define the horizontal subspace by $\mathcal{H}(\theta)=\ker(C_{\theta})$ and the vertical subspace by $\mathcal{V}(\theta)=\ker(d_{\theta}P)$. These subspaces decompose the tangent space by $T_{\theta}TM=\mathcal{H}(\theta)\oplus \mathcal{V}(\theta)$. For every $\xi,\eta\in T_{\theta}TM$, the Sasaki metric is defined by
\begin{equation}\label{sasa}
    \langle \xi,\eta \rangle_s = \langle d_{\theta}P(\xi), d_{\theta}P(\eta) \rangle_p + \langle C_{\theta}(\xi), C_{\theta}(\eta) \rangle_p.
\end{equation}
This metric induces a Riemannian distance $d_s$ usually called Sasaki distance.

For every $\theta\in T_1M$, denote by $G(\theta)\subset T_{\theta}T_1M$ the subspace tangent to the geodesic flow at $\theta$. Let $N(\theta)\subset T_{\theta}T_1M$ be the subspace orthogonal to $G(\theta)$ with respect to the Sasaki metric. For every $\theta\in T_1M$, $H(\theta)=\mathcal{H}(\theta)\cap N(\theta)$ and $V(\theta)=\mathcal{V}(\theta)\cap N(\theta)$.
From the above decomposition we have
\[  T_{\theta}T_1M=H(\theta)\oplus V(\theta)\oplus G(\theta) \quad \text{ and }\quad N(\theta)=H(\theta)\oplus V(\theta). \]
So, every $\xi\in N(\theta)$ has decomposition $\xi=(\xi_h,\xi_v)\in H(\theta)\oplus V(\theta)$. We call $\xi_h$ and $\xi_v$ the horizontal and vertical components of $\xi$ respectively.

\subsection{Horospheres and horocycles}\label{h}

In this subsection we assume that $(M,g)$ is a compact surface without conjugate points and genus greater than one. We introduce important asymptotic objects in the universal covering. We follow \cite{esch77} and part II of \cite{pesin77}. Let $\theta\in T_1\tilde{M}$ and $\gamma_{\theta}$ be the geodesic induced by $\theta$. We define the forward Busemann function by
\[ b_{\theta}: \tilde{M}\to \mathbb{R} \qquad p\mapsto b_{\theta}(p)=\lim_{t\to \infty}d(p,\gamma_{\theta}(t))-t. \]
From now on, for every $\theta=(p,v)\in T_1\Tilde{M}$ we denote $-\theta:=(p,-v)\in T_1\Tilde{M}$. The stable and unstable horosphere of $\theta$ are defined by
\[  H^+(\theta)=b_{\theta}^{-1}(0)\subset \tilde{M} \quad \text{ and }\quad  H^-(\theta)=b_{-\theta}^{-1}(0)\subset \tilde{M}. \]
We lift these horospheres to $T_1\tilde{M}$. Denote by $\nabla b_{\theta}$ the gradient vector field of $b_{\theta}$. We define the stable and unstable horocycle of $\theta$ by
\[  \mathcal{\tilde{F}}^s(\theta)=\{ (p,-\nabla_pb_{\theta}): p\in H^+(\theta) \}\quad \text{ and }\quad \mathcal{\tilde{F}}^u(\theta)= \{ (p,\nabla_pb_{-\theta}): p\in H^-(\theta) \}. \]
The horocycles project onto the horospheres by the canonical projection $\tilde{P}$. For every $\theta\in T_1\tilde{M}$, we define the stable and unstable families of horocycles by
\[ \mathcal{\tilde{F}}^s=(\mathcal{\tilde{F}}^s(\theta) )_{\theta\in T_1\tilde{M}} \quad \text{ and }\quad \mathcal{\tilde{F}}^u=( \mathcal{\tilde{F}}^u(\theta) )_{\theta\in T_1\tilde{M}}. \]
We also define the center stable and center unstable sets of $\theta$ by 
\[\mathcal{\tilde{F}}^{cs}(\theta)=\bigcup_{t\in \mathbb{R}} \mathcal{\tilde{F}}^s(\phi_t(\theta))  \quad \text{ and }\quad \mathcal{\tilde{F}}^{cu}(\theta)=\bigcup_{t\in \mathbb{R}} \mathcal{\tilde{F}}^u(\phi_t(\theta)).\]
We can define the above objects in the case of $T_1M$. For every $\theta\in T_1M$, 
\[  \mathcal{F}^*(\theta)=d\pi (\mathcal{\tilde{F}}^*(\tilde{\theta}))\subset T_1M \quad \text{ and }\quad \mathcal{F}^*=d\pi (\mathcal{\tilde{F}}^*), \quad *=s,u,cs,cu; \]
for any lift $\tilde{\theta}\in T_1\tilde{M}$ of $\theta$. Let us state some properties of these objects.
\begin{pro}[\cite{esch77,pesin77}]\label{h1}
Let $M$ be a compact surface without conjugate points of genus greater than one. Then, for every $\theta\in T_1\tilde{M}$,
\begin{enumerate}
    \item Busemann functions $b_{\theta}$ are $C^{1,L}$ with $L$-Lipschitz unitary gradient for a uniform constant $L>0$ \cite{knip86}.
    \item Horospheres $H^+(\theta),H^-(\theta)\subset \tilde{M}$ and horocycles $\mathcal{\tilde{F}}^s(\theta),\mathcal{\tilde{F}}^u(\theta)\subset T_1\tilde{M}$ are embedded curves.
    \item The families $\mathcal{\tilde{F}}^s,\mathcal{\tilde{F}}^u$ and $\mathcal{F}^s,\mathcal{F}^u$ are continuous foliations of $T_1\tilde{M},T_1M$ respectively, and invariant by the geodesic flow: for every $t\in \mathbb{R}$,    
    \begin{equation}
\Tilde{\phi}_t(\mathcal{\tilde{F}}^s(\theta))=\mathcal{\tilde{F}}^s(\Tilde{\phi}_t(\theta)).
    \end{equation}
\end{enumerate}
\end{pro}

\subsection{Morse's shadowing and consequences}\label{special}
In 1924, Morse \cite{morse24} studied a special class of geodesics of closed surfaces of genus greater than one. These surfaces always admit a metric of negative curvature called hyperbolic metric. For this hyperbolic metric, its geodesics are called hyperbolic geodesics. 
\begin{teo}[\cite{morse24}]\label{mor}
Let $(M,g)$ be a compact surface without conjugate points of genus greater than one and $\tilde{M}$ be its universal covering. Then, there exists $R>0$ such that for every geodesic $\gamma\subset \tilde{M}$ there exists a hyperbolic geodesic $\gamma'\subset \tilde{M}$ with  Hausdorff distance between $\gamma$ and $\gamma'$ bounded above by $R$.
\end{teo}
Given two geodesics $\gamma_1,\gamma_2\subset \Tilde{M}$, we say that $\gamma_1$ and $\gamma_2$ are asymptotic if $d(\gamma_1(t),\gamma_2(t))\leq C$ for every $t\geq 0$ and for some $C>0$. If the last inequality holds for every $t\in \mathbb{R}$, $\gamma_1$ and $\gamma_2$ are called bi-asymptotic. So, Theorem \ref{mor} says that $\gamma$ and $\gamma'$ are bi-asymptotic with respect to the hyperbolic distance. This gives a uniform bound between bi-asymptotic geodesics.
\begin{teo}\label{morse}
Let $(M,g)$ be a compact surface without conjugate points and genus greater than one. Then there exists $Q(M)>0$ such that the Hausdorff distance between any two bi-asymptotic geodesics is bounded above by $Q(M)$.
\end{teo}
For each $\theta\in T_1\tilde{M}$, we define the intersections 
\[ I(\theta)=H^+(\theta)\cap H^-(\theta)\subset \tilde{M} \quad \text{ and }\quad  \mathcal{\tilde{I}}(\theta)=\mathcal{\tilde{F}}^s(\theta)\cap \mathcal{\tilde{F}}^u(\theta)\subset T_1\tilde{M}.  \]
We call $\mathcal{\tilde{I}}(\theta)$ the class of $\theta$. For the canonical projection $\tilde{P}$ we have $\tilde{P}(\mathcal{\tilde{I}})=I(\theta)$.

We observe that for every $\eta=(q,w)\in \mathcal{\Tilde{I}}(\theta)$ with $q\in I(\theta)$, the geodesic $\gamma_{\eta}$ is bi-asymptotic to $\gamma_{\theta}$. So, we can translate the bounds of Theorem \ref{morse} from bi-asymptotic geodesics to intersections between horospheres and horocycles. This fact is included in the following proposition.
\begin{pro}\label{tops}
Let $M$ be a compact surface without conjugate points of genus greater than one and $\tilde{M}$ be its universal covering. Then, for every $\theta\in T_1\tilde{M}$
\begin{enumerate}
    \item $I(\theta)$ and $\mathcal{\tilde{I}}(\theta)$ are compact connected curves of $\tilde{M}$ and $T_1\tilde{M}$ respectively (Corollary 3.3 of \cite{riff18}).
    \item $Diam(I(\theta))\leq Q$ and $Diam(\mathcal{\tilde{I}}(\theta))\leq \tilde{Q}$ for some $Q(M),\tilde{Q}(M)>0$.
\end{enumerate}
\end{pro}
We remark that for every $\theta \in T_1M$ and every lift $\tilde{\theta} \in T_1\tilde{M}$ of $\theta$, we have
\[ d\pi(\mathcal{\tilde{I}}(\tilde{\theta}))=\mathcal{I}(\theta)=\mathcal{F}^s(\theta)\cap \mathcal{F}^u(\theta). \]
\begin{defi}
Let $M$ be a compact surface without conjugate points and genus greater than one. We say that $\theta\in T_1M$ is an expansive point and $\mathcal{I}(\theta)$ is a trivial class if $\mathcal{I}(\theta)$ is a single point. Otherwise, $\theta$ is a non-expansive point and $\mathcal{I}(\theta)$ is a non-trivial class. 
\end{defi}
The expansive points form the so-called expansive set 
\[   \mathcal{R}_0=\{ \theta \in T_1M: \mathcal{F}^s(\theta)\cap \mathcal{F}^u(\theta)= \{ \theta \} \}. \]
The complement of $\mathcal{R}_0$ is called the non-expansive set. In addition, note that any non-trivial class $\mathcal{I}(\theta)$ has two boundary points. 
\begin{cor}[\cite{riff18}]\label{key1}
Let $M$ be a compact surface without conjugate points of genus greater than one and $\tilde{M}$ be its universal covering. For every $\theta\in T_1\tilde{M}$, if $\eta=(q,w)\in \mathcal{\tilde{F}}^s(\theta)$ and $\gamma_{\eta}$ is bi-asymptotic to $\gamma_{\theta}$ then 
\[ \eta \in \mathcal{\Tilde{I}}(\theta)=\mathcal{\tilde{F}}^s(\theta)\cap\mathcal{\tilde{F}}^u(\theta) \quad \text{ and }\quad q\in \Tilde{I}(\theta)=H^+(\theta)\cap H^-(\theta). \]
\end{cor}

\subsection{Visibility manifolds}
This subsection introduces visibility manifolds and state some of their dynamical and geometric properties. Let $M$ be a simply connected Riemannian manifold without conjugate points. For every $x,y\in M$, denote by $[x,y]$ the geodesic segment joining $x$ to $y$. For $z\in M$ we also denote by $\sphericalangle_z(x,y)$ the angle at $z$ formed by $[z,x]$ and $[z,y]$. We say that $M$ is a visibility manifold if for every $z\in M$ and every $\epsilon>0$ there exists $R(\epsilon, z)>0$ such that 
\[ \text{ if }x,y\in M \text{ with }d(z,[x,y])>R(\epsilon, z) \quad \text{ then }\quad \sphericalangle_z(x,y)<\epsilon.\]
If $R(\epsilon,z)$ does not depend on $z$ then $M$ is called a uniform visibility manifold.
\begin{teo}[\cite{eber72}]
If $M$ is a compact surface without conjugate points of genus greater than one then its universal covering is a uniform visibility manifold.
\end{teo}
In 1973, Eberlein extended some transitivity properties of the geodesic flow to the case of compact manifolds without conjugate points. A foliation is called minimal if each of its leaves is dense.
\begin{teo}[\cite{eber72,eber73neg2}]\label{v1}
Let $M$ be a compact surface without conjugate points and genus greater than one. Then
\begin{enumerate}
    \item The horospherical foliations $\mathcal{F}^s$ and $\mathcal{F}^u$ are minimal.
    \item The geodesic flow $\phi_t$ is topologically mixing.
    \item For every $\theta,\xi \in T_1\tilde{M}$ with $\theta\not\in \mathcal{\tilde{F}}^{cu}(\xi)$ there exists $\eta_1,\eta_2\in T_1\tilde{M}$ such that
    \[ \mathcal{\tilde{F}}^s(\theta)\cap \mathcal{\tilde{F}}^{cu}(\xi)=\mathcal{\Tilde{I}}(\eta_1) \quad \text{ and }\quad  \mathcal{\tilde{F}}^s(\xi)\cap \mathcal{\tilde{F}}^{cu}(\theta)=\mathcal{\Tilde{I}}(\eta_2).\]
\end{enumerate}
\end{teo}
We can transform item (3) into intersections of unstable horocycles and central stable sets. There exist $t_1,t_2\in \mathbb{R}$ such that
\[ \mathcal{\tilde{F}}^{cs}(\theta)\cap \mathcal{\tilde{F}}^u(\xi)=\mathcal{\Tilde{I}}(\Tilde{\phi}_{t_1}(\eta_1)) \quad \text{ and }\quad \mathcal{\tilde{F}}^{cs}(\xi)\cap\mathcal{\tilde{F}}^u(\theta)=\mathcal{\Tilde{I}}(\Tilde{\phi}_{t_2}(\eta_2)).  \]
The above intersections are called the heteroclinic connections of $\Tilde{\phi}_t$. 

Recall that for a compact manifold of negative curvature, its geodesic flow is uniformly hyperbolic \cite{anos67}. This provides invariant submanifolds (which agree with the horocycles) with hyperbolic behavior. However, for a general compact manifold without conjugate points, its geodesic flow might not be uniformly hyperbolic. Despite this, the horocycles still have some weak hyperbolicity. From Equation \eqref{sasa} in Subsection \ref{m}, we recall that $d_s$ is the Sasaki distance restricted to $T_1\Tilde{M}$.
\begin{pro}[\cite{eber72}]\label{quesi}
Let $M$ be a compact surface without conjugate points of genus greater than one and $\tilde{M}$ be its universal covering. Then, there exist $A,B>0$ such that for every $\theta\in T_1\tilde{M}$ and every $\eta\in \mathcal{\tilde{F}}^s(\theta)$, 
\[   d_s(\Tilde{\phi}_t(\theta),\Tilde{\phi}_t(\eta))\leq Ad_s(\theta,\eta)+B, \quad \text{ for every } t\geq 0.  \]
\end{pro}

\subsection{Some dynamical and ergodic properties of continuous flows on compact metric spaces}\label{p5}
We first introduce the dynamical properties. Let $\psi_t:X\to X$ be a continuous flow acting on a compact metric space $X$. We say that $\psi_t$ is topologically transitive if there exists a dense orbit. The flow $\psi_t$ is topologically mixing if for every open sets $A,B\subset X$ there exists $t_0>0$ such that $\psi_t(A)\cap B\neq \emptyset$ for $|t|\geq t_0$. For every $x\in X$ and every $\epsilon>0$, we define the
\[ \text{strong stable set of $x$: }W^{ss}(x)=\{ y\in X:d(\psi_t(x),\psi_t(y))\to 0 \text{ as }t\to \infty\}, \]
\[  \epsilon\text{-strong stable set: }W^{ss}_{\epsilon}(x)=\{ y\in W^{ss}(x):d(\psi_t(x),\psi_t(y))\leq \epsilon, \text{ for every }t\geq 0\}. \]
The strong unstable $W^{uu}(x)$ and $\epsilon$-strong unstable $W^{uu}_{\epsilon}(x)$ sets are defined similarly for $t\leq0$. 

The flow $\psi_t$ has a local product structure if for every $\epsilon>0$ there exists $\delta>0$ such that if $x,y\in X$ satisfy $d(x,y)\leq \delta$ then there exists a unique $\tau\in \mathbb{R}$ with 
\[ |\tau|\leq \epsilon \quad \text{ and } \quad W^{ss}_{\epsilon}(x)\cap W^{uu}_{\epsilon}(\phi_{\tau}(y))\neq \emptyset.\]
The orbit of the intersection point follows the orbit of $x$ in the future and the orbit of $y$ in the past.

The flow $\psi_t$ is expansive if there exists $\epsilon>0$ such that if $x,y\in X$ satisfy
\[   d(\phi_t(x),\phi_{\rho(t)}(y))\leq \epsilon \text{ for every }t \in \mathbb{R}\]
and some reparametrization $\rho$, then there exists $\tau\in [-\epsilon,\epsilon]$ with $y=\phi_{\tau}(x)$. We call $\epsilon$ a constant of expansivity of $\phi_t$. 

In the context of continuous flows without singularities acting on compact manifolds, the above definition is equivalent to Bowen-Walters expansivity definition \cite{bowen72}. We remark that Anosov flows are always expansive.

Let us define a special kind of semi-conjugacy between flows. Let $\phi_t:Y\to Y$ and $\psi_t:X\to X$ be two continuous flows acting on compact topological spaces. A map $\chi:Y\to X$ is called a time-preserving semi-conjugacy if $\chi$ is a continuous surjection satisfying $\chi\circ\phi_t=\psi_t\circ \chi$ for every $t\in \mathbb{R}$. In this case, we say that $\psi_t$ is time-preserving semi-conjugate to $\phi_t$ or is a time-preserving factor of $\phi_t$.

We now give the ergodic properties. Let $\psi_t:X\to X$ be a continuous flow acting on a compact metric space $(X,d)$. A Borel set $Z\subset X$ is invariant by the flow if $\psi_t(Z)=Z$ for every $t\in \mathbb{R}$. A probability measure $\nu$ on $X$ is invariant by the flow if $(\psi_t)_*\nu=\nu$ for every $t\in \mathbb{R}$. Denote by $\mathcal{M}(\psi)$ the set of all flow-invariant-measures on $X$. A measure $\nu\in \mathcal{M}(\psi)$ is ergodic if for every flow-invariant set $A\subset X$, we have either $\nu(Z)=0$ or $\nu(Z)=1$.

Let $Z\subset X$ be a flow-invariant Borel set and $\nu$ be a flow-invariant measure supported on $Z$. We define the metric entropy $h_{\nu}(\psi,Z)$ of $\nu$ with respect to the flow $\psi$ as the metric entropy $h_{\nu}(\psi_1,Z)$ with respect to its time-1 map $\psi_1$ \cite{walt00}. For $Z=X$ we write $h_{\nu}(\psi)$. When $Z$ is also compact, we define the topological entropy of $Z$ as follows. For every $\epsilon,T>0$ and every $x\in Z$, we define the $(T,\epsilon)$-dynamical balls by
\begin{equation}\label{dyna}
    B(x,\epsilon,T)=\{ y\in Z:d(\psi_s(x),\psi_s(y))<\epsilon,s\in [0,T] \}
\end{equation}
Denote by $M(T,\epsilon,Z)$ the minimum cardinality of any cover of $Z$ by $(T,\epsilon)$-dynamical balls. The topological entropy of $Z$ with respect to $\psi$ is
\[  h(\psi,Z)=\lim_{\epsilon\to 0}\limsup_{T\to \infty} \frac{1}{T} \log M(T,\epsilon,Z).\]
For $Z=X$ we write $h(\psi)$. We remark that $h(\psi,Z)=h(\psi_1,Z)$ where $h(\psi_1,Z)$ is the topological entropy of $Z$ with respect to the time-1 map $\psi_1$. For each flow-invariant compact set $Z\subset X$, the variational principle \cite{dina71} says 
\begin{equation}\label{varcon}
h(\psi,Z)= \sup_{\nu} h_{\nu}(\psi,Z),    
\end{equation}
where $\nu$ varies over all flow-invariant measures supported on $Z$. We say that $\mu\in \mathcal{M}$ supported on $Z$ is a measure of maximal entropy if $h_{\mu}(\psi,Z)$ achieves the maximum in \eqref{varcon}. If $Z=X$ and $\mu$ is the only measure satisfying this condition then $\mu$ is the unique measure of maximal entropy for the flow $\psi$. 

\section{The quotient flow}\label{s1}
In what follows we will assume that $(M,g)$ is a compact surface without conjugate points and genus greater than one. The main idea of the section is to built a factor flow of the geodesic flow of $(M,g)$. We follow the constructions of Gelfert and Ruggiero \cite{gelf19}.

The properties given in Subsection \ref{special} suggest that an equivalence relation identifying each curve $\mathcal{I}(\theta)$ with a single point $[\theta]$ will do the job. Two points $\theta, \eta \in T_1M$ are equivalent, $\theta \sim \eta$, if and only if $\eta\in \mathcal{F}^s(\theta)$ and for every $\tilde{\theta},\tilde{\eta}\in T_1\tilde{M}$ lifts of $\theta,\eta$ with $\tilde{\eta}\in \mathcal{\tilde{F}}^s(\tilde{\theta})$, it holds that $\gamma_{\tilde{\theta}}$ and $\gamma_{\tilde{\eta}}$ are bi-asymptotic. 
\begin{lem}\label{char1}
For every $\eta,\theta\in T_1M$, $\eta\sim\theta$ if and only if $\eta\in \mathcal{I}(\theta)$.
\end{lem}
\begin{proof}
If $\eta\sim\theta$ then there exist lifts $\tilde{\eta},\tilde{\theta}$ of $\eta,\theta$ such that $\tilde{\eta}\in\mathcal{\tilde{F}}^s(\tilde{\theta})$ and $\gamma_{\tilde{\eta}},\gamma_{\tilde{\theta}}$ are bi-asymptotic. Corollary \ref{key1} says that $\tilde{\eta}\in \mathcal{\tilde{I}}(\tilde{\theta})$ and projecting by $d\pi$ we get $\eta\in \mathcal{I}(\theta)$. The reverse implication is straightforward.
\end{proof}
This lemma and the properties of the horospherical leaves guarantee the above relation is an equivalence relation. This relation induces a quotient space $X$ and a quotient map 
\[  \chi:T_1M\to X \qquad \theta\mapsto \chi(\theta)=[\theta],\]
where $[\theta]$ is the equivalence class of $\theta$. The geodesic flow and the quotient map induce a quotient flow 
\[   \psi:\mathbb{R}\times X\to X \qquad (t,[\theta])\mapsto \psi_t[\theta]=[\phi_t(\theta)]=\chi \circ \phi_t(\theta).\]
We endow $X$ with the quotient topology. We state below some properties of these new objects.
\begin{lem}
Let $M$ be a compact surface without conjugate points of genus greater than one. Then,
\begin{enumerate}
    \item The quotient space $X$ is a compact space.
    \item The quotient map $\chi$ is a time-preserving semi-conjugacy hence $\psi_t$ is a time-preserving factor of $\phi_t$: for every $t\in \mathbb{R}$,
    \begin{equation}\label{semi}
        \chi\circ \phi_t = \psi_t \circ \chi.
    \end{equation}
\end{enumerate}
\end{lem}
\begin{proof}
For item 1, from the definition of the quotient topology we see that $\chi$ is a continuous surjection. Since $T_1M$ is compact so is $X$.

For item 2, let us first show that $\psi_t$ is well-defined. Let $[\eta]\in X$ and $\xi\in T_1M$ be any representative of the equivalence class of $\eta$. So, we have $\xi\sim \eta$ hence $\xi\in \mathcal{I}(\eta)=\mathcal{F}^s(\eta)\cap \mathcal{F}^u(\eta)$ by Lemma \ref{char1}. By the invariance of the horospherical foliations it follows that $\phi_t(\xi)\in \mathcal{F}^s(\phi_t(\eta))\cap \mathcal{F}^u(\phi_t(\eta))=\mathcal{I}(\phi_t(\eta))$ for every $t\in\mathbb{R}$. This means that $\phi_t(\xi)\sim \phi_t(\eta)$ and hence $\psi_t[\xi]=[\phi_t(\xi)]=[\phi_t(\eta)]=\psi_t[\eta]$. Since $\chi$ is a continuous surjection, the item follows from the definition of the quotient flow.
\end{proof}
The following concept is useful for defining certain open sets of the quotient topology.
\begin{defi}
Let $A$ be a subset of $T_1M$. We say that $A$ is saturated with respect to $\chi$, or simply saturated, if $\chi^{-1}\circ \chi(A)=A$.
\end{defi}
\begin{lem}
Let $M$ be a compact surface without conjugate points of genus greater than one. Then,
\begin{enumerate}
    \item For every $\eta \in T_1M$, $\chi^{-1}\circ \chi (\eta)=\mathcal{I}(\eta)$.
    \item For every open saturated set $U\subset T_1M$, $\chi(U)$ is an open set in $X$.
\end{enumerate}
\end{lem}
\begin{proof}
For item 1, $\chi^{-1}\circ \chi(\eta)=\chi^{-1}[\eta]$ implies that $\chi^{-1}\circ \chi(\eta)$ is the equivalence class of $\eta$ seen as subset of $T_1M$. By Lemma \ref{char1}, this equivalence class agree exactly with $\mathcal{I}(\eta)$. For item 2, 
by definition $\chi(U)$ is open in the quotient topology if $\chi^{-1}(\chi(U))$ is open in $T_1M$. The result follows since $\chi^{-1}(\chi(U))=U$.
\end{proof}
We extend the above construction to the geodesic flow $\tilde{\phi}_t$ of $(\tilde{M}, \tilde{g})$. Following Lemma \ref{char1}, we say that $\eta,\theta\in T_1\tilde{M}$ are equivalent if and only if $\eta\in \mathcal{\tilde{I}}(\theta)$. As before, the equivalence relation analogously induces a quotient space $\tilde{X}$, a quotient map $\tilde{\chi}$ and a quotient flow $\tilde{\psi}_t$. In a similar way, $\tilde{\chi}$ is a time-preserving semi-conjugacy between $\tilde{\phi}_t$ and $\tilde{\psi}_t$. Moreover, since $T_1\tilde{M}$ is a covering space of $T_1M$ with covering map $d\pi$, we see that $\tilde{X}$ is also a covering space of $X$ with a covering map $\Pi$ induced by $d\pi$:
\begin{equation}\label{cove}
    \Pi:\tilde{X}\to X, \qquad [\Tilde{\eta}]\mapsto \Pi[\Tilde{\eta}]=\chi\circ d\pi(\Tilde{\eta}).
\end{equation}
It is not hard to show that $\Pi$ is a well-defined covering map using the map $d\pi$. 

\section{Topological properties of the quotient flow}\label{s3}
In this section we build a special basis of neighborhoods of the quotient topology of $X$. We extend a construction made by Gelfert and Ruggiero for the case of compact higher genus surfaces without focal points (Section 4 of \cite{gelf19}). As an application, we show that $X$ is a compact metrizable space.

We highlight that for every $\eta\in T_1M$, the set $\mathcal{F}^s(\eta)$ ($\mathcal{F}^u(\eta)$) is the orbit of a continuous complete flow: the stable (unstable) horocycle flow. The same holds for $\mathcal{\tilde{F}}^s(\tilde{\eta}),\mathcal{\tilde{F}}^u(\tilde{\eta})$ for every lift $\tilde{\eta}\in T_1\tilde{M}$ of $\eta$. Each of such sets is a Lipschitz embedded curve and can be parametrized by arc-length $c^s_{\eta}: \mathbb{R}\to \mathcal{F}^s(\eta), c^u_{\eta}: \mathbb{R}\to \mathcal{F}^u(\eta),\tilde{c}^s_{\eta}: \mathbb{R}\to \mathcal{\tilde{F}}^s(\tilde{\eta})$ and $\tilde{c}^u_{\eta}: \mathbb{R}\to \mathcal{\tilde{F}}^u(\tilde{\eta})$. In particular, each connected subset of $\mathcal{F}^{s,u}(\eta),\mathcal{\tilde{F}}^{s,u}(\tilde{\eta})$ is homeomorphic to an interval.

The construction of a basis for the quotient space requires a better understanding of the set of expansive points $\mathcal{R}_0\subset T_1M$. Assuming no focal points or continuous Green bundles, Gelfert and Ruggiero showed that $\mathcal{R}_0$ is open and dense. This might no longer be the case if we drop these hypothesis. Let us start with the following elementary lemma.
\begin{lem}\label{int}
Any real interval cannot be the union of disjoint closed intervals.
\end{lem}
\begin{lem}\label{que2}
Let $\theta\in T_1\Tilde{M}$ be a non-expansive point and $\mathcal{\tilde{I}}(\theta)$ be its non-trivial class. Then the boundary points of $\mathcal{\tilde{I}}(\theta)$ are accumulated by expansive points.
\end{lem}
\begin{proof}
Let $c:\mathbb{R}\to \mathcal{\tilde{F}}^s(\theta)$ be an arc-length parametrization of $\mathcal{\tilde{F}}^s(\theta)$. Let $a,b\in \mathbb{R}$ with $a<b$ such that $\mathcal{\tilde{I}}(\theta)=c([a,b])$. Suppose that the boundary point $b$ is not accumulated by expansive points. Then, there exists $\delta>0$ such that $(b,b+\delta)$ has no expansive points. So, either $(b,b+\delta)$ is a single class or it is a disjoint union of distinct classes. Since classes are closed, Lemma \ref{int} shows that $[b,b+\delta]$ is a subset of a single class. So, $b$ is a common point of both classes hence $[a,b+\delta]$ is a single class, a contradiction to $b$ be a boundary point.
\end{proof}
Next, for each $\theta \in T_1\tilde{M}$, let us define a family of open neighborhoods $A_i$ of $\mathcal{\tilde{I}}(\theta)$ such that $\tilde{\chi}(A_i)$ are open neighborhoods of $\tilde{\chi}(\theta)$ as well. For every $\delta,\epsilon>0$, there exist $a,b\in \mathbb{R}$ and a map
\[  R: (a-\epsilon, b+\epsilon)\times (-\delta,\delta)\to T_1\tilde{M}, \qquad \text{satisfying the conditions: }\]
\begin{enumerate}
    \item Let $(0,s)\mapsto R(0,s)$ be the arc-length parametrization of a $\delta$-neighborhood of $\theta$ in $V(\theta)$ with respect to the Sasaki metric where $V(\theta)$ is the vertical subspace of $\theta$.
    \item For each $s_0\in (-\delta,\delta)$, let $(r,s_0)\mapsto R(r,s_0)$ be the arc-length parametrization of the continuous curve $R(\cdot,s_0)$ in $\mathcal{\tilde{F}}^s(R(0,s_0))$.
    \item $R([a,b],0)=\mathcal{\tilde{I}}(\theta)$, $R(0,0)=\theta$ and $(r,0)\mapsto R(r,0)$ be the arc-length parametrization of a $\epsilon$-neighborhood of $\mathcal{\tilde{I}}(\theta)$ in $\mathcal{\tilde{F}}^s(\theta)$.
\end{enumerate}
The image of $R$ is denoted by $\Sigma= \Sigma(\theta,\epsilon,\delta)= R((a-\epsilon, b+\epsilon)\times (-\delta,\delta))$. The continuity of the horospherical foliations ensures that $R$ is a homeomorphism and $\Sigma$ is a 2-dimensional section containing $\mathcal{\tilde{I}}(\theta)$. Note that $\Sigma$ is foliated by stable horospherical leaves of points in $V(\theta)$. Since these leaves are topologically transverse to the geodesic flow, $\Sigma$ is a cross section. For $\tau>0$, Brouwer's open mapping Theorem gives the following open 3-dimensional neighborhood of $\mathcal{\tilde{I}}(\theta)$:
\[  B=B(\theta,\epsilon,\delta,\tau)= \bigcup_{|t|<\tau}\tilde{\phi}_t(\Sigma). \]
We next define a projection map $Pr:B\to \Sigma$. For every $\eta\in B$, $Pr(\eta)$ is the projection of $\eta$ along the geodesic flow $\tilde{\phi}_t$. From the properties of $\tilde{\phi}_t$, we see that $Pr$ is a continuous surjection. For every $\eta\in \Sigma$, we define the stable (unstable) interval and their intersection
\[  W^s(\eta)=\mathcal{\tilde{F}}^s(\eta)\cap \Sigma, \quad W^u(\eta)=Pr(\mathcal{\tilde{F}}^u(\eta)\cap B) \quad \text{ and }\quad [\xi,\eta]=W^s(\xi)\cap W^u(\eta).\]
Since $\Sigma$ is foliated by the stable horospherical foliation $\mathcal{\Tilde{F}}^s$, we see that $W^s$ is exactly $\mathcal{\Tilde{F}}^s$ while $W^u$ is not the unstable horospherical foliation $\mathcal{\Tilde{F}}^u$ because the geodesic flow cannot have a local section that is foliated by both $\mathcal{\Tilde{F}}^s$ and $\mathcal{\Tilde{F}}^u$.

To build a new cross section let us choose four expansive points $\theta_1, \theta_2, \eta_1,\eta_2\in\Sigma$. Since $\mathcal{\tilde{I}}(\theta)$ is non-trivial, Lemma \ref{que2} says that for $\epsilon>0$ there exists $c\in (a-\epsilon,a), d\in (b,b+\epsilon)$ such that $\theta_1=R(c,0)$ and $\theta_2=R(d,0)$ are expansive points of $W^s(\theta)$. This also implies that
\begin{equation}\label{inclu}
 \mathcal{\tilde{I}}(\theta)=R([a,b],0)\subset R([c,d],0)   
\end{equation}
We define the upper and lower region of $\Sigma$ by
\[\Sigma_+=\{ R(r,s): r\in (a-\epsilon,b+\epsilon), s>0\} \text{ and } \Sigma_-=\{ R(r,s): r\in (a-\epsilon,b+\epsilon), s<0\}.\]
Pick some expansive points $\eta_1\in W^u(\theta_1)\cap\Sigma_+$ and $\eta_2\in W^u(\theta_2)\cap\Sigma_-$. The new cross section $U= U(\theta,\epsilon, \delta,\theta_1, \theta_2, \eta_1,\eta_2)\subset \Sigma$ is the open 2-dimensional region in $\Sigma$ bounded by $W^u(\theta_1)$, $W^u(\theta_2)$, $W^s(\eta_1)$ and $W^s(\eta_2)$. 
\begin{lem}
The open cross section $U\in \Sigma$ is a saturated set containing $\mathcal{\tilde{I}}(\theta)$.
\end{lem}
\begin{proof}
From relation \eqref{inclu} we see that $\mathcal{\tilde{I}}(\theta)$ is surrounded by $\theta_1,\theta_2$ in $W^s(\theta)$ hence it is included in $U$. Now, suppose by contradiction that there exists a non-expansive $\eta \in U$ such that $\mathcal{\tilde{I}}(\eta)$ is not included in $U$. This implies that $\mathcal{\tilde{I}}(\eta)$ intersects the boundary of $U$ at some $\xi$. Since $\eta\in \mathcal{\tilde{I}}(\xi)$ it follows that $\eta\in W^s(\xi)\cap W^u(\xi)$. Thus $\eta$ belongs to the boundary of $U$, a contradiction.
\end{proof}
As above, for $\tau>0$ Brouwer's open mapping Theorem says that
\[  A=A(\theta,\epsilon, \delta,\tau, \theta_1, \theta_2, \eta_1,\eta_2)= \bigcup_{|t|<\tau}\tilde{\phi}_t(U) \]
is an open 3-dimensional neighborhood of $\mathcal{\tilde{I}}(\theta)$. Since $U$ is saturated so is $A$.
Thus, for every $\theta\in T_1\tilde{M}$, we have built a family 
\[ \{ A(\theta,\epsilon, \delta,\tau, \theta_1, \theta_2, \eta_1,\eta_2): \epsilon, \delta,\tau>0, \theta_1, \theta_2\in W^s(\theta), \eta_1\in W^u(\theta_1),\eta_2\in W^u(\theta_2) \} \]
of saturated neighborhoods of $\mathcal{\tilde{I}}(\theta)$. Consider the family of quotients of sets
\[   \{ \tilde{\chi}( A(\theta,\epsilon, \delta,\tau, \theta_1, \theta_2, \eta_1,\eta_2)): \epsilon, \delta,\tau>0, \theta_1, \theta_2\in W^s(\theta), \eta_1\in W^u(\theta_1),\eta_2\in W^u(\theta_2) \}. \]
\begin{lem}\label{que1}
For every $\theta\in T_1\tilde{M}$, the family
\[   \mathcal{A}_{\theta}= \{ \tilde{\chi}( A(\theta,\epsilon_l, \delta_m,\tau_n)): \epsilon_l=1/l,\delta_m=1/m,\tau_n=1/n \text{ with } l,m,n\in \mathbb{N} \} \]
is a countable basis of neighborhoods of $[\theta]\in \tilde{X}$. Hence $\tilde{X}$ is first countable and $\{ \mathcal{A}_{\theta}: \theta\in T_1\tilde{M} \}$ is a basis for the quotient topology of $\tilde{X}$. 
\end{lem}
\begin{proof}
For every $\theta\in T_1\tilde{M}$, we know that $A=A(\theta,\epsilon, \delta,\tau, \theta_1, \theta_2, \eta_1,\eta_2)$ is an open neighborhood of $\mathcal{\tilde{I}}(\theta)$. Since $A$ is saturated it follows that $\tilde{\chi}^{-1}\circ \tilde{\chi}(A)=A$. We see that $\tilde{\chi}(A)\subset \tilde{X}$ is an open set containing $[\theta]$ and $\{ \tilde{\chi}( A(\theta,\epsilon, \delta,\tau, \theta_1, \theta_2, \eta_1,\eta_2))\}$ is a family of open neighborhoods of $[\theta]\in X$. 

Choosing $\epsilon, \delta, \tau>0$ small enough, every neighborhood $V$ of $\mathcal{\tilde{I}}(\theta)$ contains some $A(\theta,\epsilon, \delta,\tau, \theta_1, \theta_2, \eta_1,\eta_2)$. Given an open set $U\subset X$ containing $[\theta]$, $\tilde{\chi}^{-1}(U)$ is an open neighborhood of $\mathcal{\tilde{I}}(\theta)$. So, there exists $A(\theta,\epsilon, \delta,\tau, \theta_1, \theta_2, \eta_1,\eta_2)\subset \tilde{\chi}^{-1}(U)$ and hence $\tilde{\chi}(A(\theta,\epsilon, \delta,\tau, \theta_1, \theta_2, \eta_1,\eta_2))\subset U$. Therefore the collection $\{ \tilde{\chi}( A(\theta,\epsilon, \delta,\tau, \theta_1, \theta_2, \eta_1,\eta_2)) \}$ is a basis of neighborhoods of $[\theta]\in X$. 

This property is not affected by specific choices of parameters $\theta_1,\theta_2,\eta_1,\eta_2\in \Sigma$ but by the parameters $\epsilon,\delta,\tau>0$. Choosing $\epsilon_l=1/l,\delta_m=1/m ,\tau_n=1/n$ with $l,m.n\in \mathbb{N}$ we still have $\mathcal{A}_{\theta}=\{\tilde{\chi}( A(\theta,\epsilon_l, \delta_m,\tau_n)): l,m,n\in \mathbb{N} \}$ is a basis of neighborhoods of $[\theta]$. So, $\mathcal{A}_{\theta}$ is a countable basis of neighborhoods of $[\theta]$. 
\end{proof} 
This basis is important because it provides an explicit description of basic open sets of the quotient topology. So far, we have a family of neighborhoods for every $\mathcal{\tilde{I}}(\theta)\subset T_1\tilde{M}$ and a basis of neighborhoods of $[\theta]\in \tilde{X}$. From Equation \eqref{cove} in Section \ref{s1}, we recall that $d\pi:T_1\Tilde{M}\to T_1M$ and $\Pi:\Tilde{X}\to X$ are covering maps. Projecting the above families of open neighborhoods by $\Pi$ and $d\pi$ respectively, we get corresponding families of open neighborhoods for every $\mathcal{I}(\theta)\subset T_1M$ and every $[\theta]\in X$. Thus, $X$ is first countable and $\{ \Pi(\mathcal{A}_{\theta}): \theta\in T_1\tilde{M} \}$ is a basis for the quotient topology of $X$. 

To show the metrizability of $X$ we recall a basic topological result \cite{will12}.
\begin{pro}\label{metri}
If $f:X\to Y$ is a continuous surjection from a compact metric space onto a Hausdorff space, then $Y$ is metrizable.
\end{pro}
\begin{lem}
Let $M$ be a compact surface without conjugate points and genus greater than one. Then, the quotient space $X$ is a compact metrizable space.
\end{lem}
\begin{proof}
Since $\chi$ is a continuous surjection, $X$ is compact. We next show that $X$ is Hausdorff. Choose two distinct points $[\theta],[\eta]\in X$ and suppose that $\mathcal{F}^s(\theta)\cap \mathcal{F}^s(\eta)=\emptyset$. Choosing $\delta$ small enough in Lemma \ref{que1}, we can build disjoint basic open sets because $\mathcal{F}^s$ is a foliation of $T_1M$. We now suppose that $\mathcal{F}^s(\theta)\cap \mathcal{F}^s(\eta)\neq\emptyset$ hence $\mathcal{F}^s(\theta)= \mathcal{F}^s(\eta)$. Choosing $\epsilon$ small enough in Lemma \ref{que1}, the basic open sets of $\theta$ and $\eta$ are disjoint because $\mathcal{F}^u$ is a foliation of $T_1M$. The result follows from the application of Theorem \ref{metri} to our case.
\end{proof}
So, we choose some distance on $X$ that is compatible with the quotient topology. We denote it by $d$ and called it the quotient distance. 

\section{Topological dimension of the quotient space}
This section is devoted to show that the topological dimension of the quotient space is at least two. We first define the topological dimension \cite{munk00}. Let $X$ be a topological space and $\mathcal{U}$ be an open cover of $X$. The order of $\mathcal{U}$ is the smallest $n\in \mathbb{N}$ such that every $x\in X$ belongs to at most $n$ sets of $\mathcal{U}$. An open refinement of $\mathcal{U}$ is another open cover, each of whose sets is a subset of a set in $\mathcal{U}$. The topological dimension of $X$ is the minimum $n$ such that every $\mathcal{U}$ has an open refinement of order $n+1$ or less. We have as standard examples the open sets of $\mathbb{R}^n$. For every open set $U\subset \mathbb{R}^n$, the topological dimension of $U$ is $n$.
\begin{teo}[\cite{will12}]\label{will}
Let $f:X\to Y$ be a homeomorphism between topological spaces. Then, the topological dimension of $X$ and $Y$ are equal. 
\end{teo}
Let $X$ be a topological space, $f:X\to \mathbb{R}^2$ be a continuous map and $U\subset X$ be an open set. We say that $U$ is a topological surface if the restriction of $f$ to $U$ is a homeomorphism. Theorem \ref{will} implies that every topological surface has topological dimension 2. Let $\tilde{X}$ and $X$ be the quotient spaces defined in Section \ref{s1}. To estimate their topological dimensions we find a topological surface passing through every point. 
\begin{lem}
Let $M$ be a compact surface without conjugate points and genus greater than one. Then, for every $[\theta]\in \tilde{X}$ there exists a topological surface $S_{[\theta]}$ containing $[\theta]$. In particular, $\tilde{X}$ and $X$ have topological dimension at least two. 
\end{lem}
\begin{proof}
Let $\theta \in T_1\tilde{M}$ and $V_{\theta}$ be the vertical fiber of $\theta$. Using the geodesic flow we define the set 
\[   W_{\theta}=\bigcup_{t\in \mathbb{R}}\phi_t(V_{\theta})\subset T_1\tilde{M}. \]
Since $V_{\theta}$ is homeomorphic to the circle $S^1$, $W_{\theta}$ is homeomorphic to a cylinder hence $W_{\theta}$ is a topological surface. The divergence of geodesic rays guarantees that for every $\eta,\xi\in V_{\theta}$, $\eta\not\in \mathcal{\tilde{I}}(\xi)$. So, the restriction of $\tilde{\chi}$ to $W_{\theta}$ is injective hence bijective onto its image. This implies that $\tilde{\chi}(W_{\theta})\subset \tilde{X}$ is homeomorphic to a cylinder and the result follows. Theorem \ref{will} implies that the topological dimension of $\tilde{X}$ is at least two. This conclusion extends to $X$ because $\tilde{X}$ and $X$ are locally homeomorphic.
\end{proof}
In \cite{gelf19,gelf20}, Gelfert and Ruggiero showed that $X$ and $\tilde{X}$ are topological 3-manifolds for: compact surfaces without focal points of genus greater than one and compact surfaces without conjugate points, genus greater than one and continuous Green bundles. It is not known whether the dimension of the quotient space is 3 without assuming any of the above two conditions.

\section{Topological dynamics of the quotient flow}\label{s4}
Section \ref{s1} defines a quotient model: a continuous quotient flow $\psi_t:X\to X$ time-preserving semi-conjugate to the geodesic flow $\phi_t$. Our goal is to show that $\psi_t$ has typical properties of hyperbolic topological dynamics like expansivity, local product structure and topological mixing. Notice that the geodesic flow $\phi_t$ may not be expansive due to the presence of non-trivial strips. 
\begin{teo}\label{propd}
Let $M$ be a compact surface without conjugate points of genus greater than one and $\psi_t:X\to X$ be the quotient flow. Then, $\psi_t$ is topologically mixing, expansive and has a local product. Moreover $\psi_t$ has the pseudo-orbit tracing and specification properties. 
\end{teo}

We will prove Theorem \ref{propd} in several steps. Since the geodesic flow $\phi_t$ is topologically mixing (Theorem \ref{v1}) and $\chi$ is a continuous time-preserving semi-conjugacy, we deduce that $\psi_t$ is also topologically mixing. 

Recall that for every $\eta\in T_1M$, Lemma \ref{que1} gives a relationship between neighborhoods of $[\eta]\in X$ and special neighborhoods of $\mathcal{I}(\eta)$. We use this basic open sets to get a relationship between Sasaki distance $d_s$ (Equation \eqref{sasa}) and the quotient distance $d$.
\begin{lem}\label{lev}
Let $Q>0$ be the Morse's constant given in Theorem \ref{morse}. Then, there exist $r_0,s_0>0$ such that for every $[\xi],[\eta]\in X$ with $d([\xi],[\eta])\leq r_0$ then
\[    d_s(\tilde{\xi},\tilde{\eta})\leq Q+s_0, \]
for some lifts $\tilde{\xi},\tilde{\eta}\in T_1\tilde{M}$ of $\xi,\eta\in T_1M$.
\end{lem}
\begin{proof}
We consider the basic open sets $A(\eta,\epsilon,\delta,\tau)$ provided by Lemma \ref{que1}. For every $\theta\in T_1M$, choose $\epsilon,\delta,\tau>0$ small enough so that $A(\theta,\epsilon,\delta,\tau)$ is evenly covered by $d\pi$. Clearly, the family $\mathcal{A}=\{ \chi(A(\theta,\epsilon,\delta,\tau)): \theta \in T_1M\}$ is an open cover of the compact space $X$. Let $r_0>0$ be a Lebesgue number of $\mathcal{A}$. Thus, for every $[\eta],[\xi]\in X$ with $d([\eta],[\xi])\leq r_0$, there exists $\theta\in T_1M$ such that 
\[  [\xi]\in B([\eta],r_0)\subset \chi(A(\theta,\epsilon,\delta,\tau))\in \mathcal{A}, \]
where $B([\eta],r_0)$ is the $r_0$-closed ball centered at $[\eta]$. Since $A(\theta,\epsilon,\delta,\tau)$ is evenly covered by $d\pi$, for every lift $\tilde{\theta}$ of $\theta$ there exist lifts $\tilde{A}(\tilde{\theta},\epsilon,\delta,\tau)$ and $\tilde{\eta},\tilde{\xi}\in \tilde{A}(\tilde{\theta},\epsilon,\delta,\tau)$
of $A(\theta,\epsilon,\delta,\tau),\eta,\xi$ respectively. As $\epsilon, \delta, \tau$ are small enough, there exists $s_0>0$ so that $Diam(\tilde{A}(\tilde{\theta},\epsilon,\delta,\tau))\leq Q+s_0$ hence $d_s(\tilde{\eta},\tilde{\xi})\leq Q+s_0$.
\end{proof}

\begin{lem}\label{est1}
For every $\epsilon>0$ there exists $\delta>0$ such that if $[\xi]\in X$ satisfy $d([\xi],\psi_{\tau}[\xi])\leq \delta$ for some $\tau\in \mathbb{R}$, then $|\tau|\leq \epsilon$.
\end{lem}
\begin{proof}
By contradiction, suppose there exist $\epsilon_0>0$ and sequences $[\xi_n]\in X, \tau_n\in \mathbb{R}$ such that for every $n\geq 1$
\begin{equation}\label{est}
d([\xi_n],\psi_{\tau_n}[\xi_n])\leq \frac{1}{n}  \quad \text{ and }\quad |\tau_n|\geq \epsilon_0.
\end{equation}
Up to subsequences, we can assume that $\tau_n\to T$ and $[\xi_n]\to [\xi]$. Since $\psi_t$ is continuous, $\psi_{\tau_n}[\xi_n]\to \psi_T[\xi]$. On the other hand inequalities \eqref{est} say that $[\xi_n]$ and $\psi_{\tau_n}[\xi_n]$ converge to the same limit $[\xi]=\psi_T[\xi]$. This holds if and only if $T=0$. Thus $\tau_n\to 0$, which contradicts inequalities \eqref{est}.
\end{proof}
\begin{lem}\label{exp1}
The quotient flow $\psi_t$ is expansive. 
\end{lem}
\begin{proof}
Let $r_0>0$ be given by Lemma \ref{lev}. We first show that if there are two orbits of $\psi_t$ having Hausdorff distance bounded by $r_0$, then the orbits agree. Let $[\eta],[\xi]\in X$ with $d(\psi_t[\eta],\psi_{\rho(t)}[\xi])\leq r_0$ for every $t\in \mathbb{R}$ and some reparametrization $\rho$. By Lemma \ref{lev}, there exist lifts $\tilde{\eta},\tilde{\xi}\in T_1\tilde{M}$ of $\eta,\xi$ such that
\[  d_s(\phi_t(\tilde{\eta}),\phi_{\rho(t)}(\tilde{\xi}))\leq Q+s_0, \text{ for every } t\in \mathbb{R}.  \]
Thus, the orbits of $\tilde{\eta}$ and $\tilde{\xi}$ have Hausdorff distance bounded by $Q+s_0$ hence the orbits are bi-asymptotic. This implies that there exists $\tau\in \mathbb{R}$ so that $\tilde{\xi}\in \mathcal{\tilde{I}}(\phi_{\tau}(\tilde{\eta}))$ hence $[\xi]=\psi_{\tau}[\eta]$.

Given $\epsilon>0$, Lemma \ref{est1} yields $\delta_1>0$ satisfying its statement. Let $\delta=\min(\delta_1,r_0)$. If the orbits of $[\eta]$ and $[\xi]$ have Hausdorff distance bounded by $\delta\leq r_0$ then $[\xi]=\psi_{\tau}[\eta]$ and $|\tau|\leq \epsilon$ because $d([\eta],\psi_{\tau}[\eta])=d([\eta],[\xi])\leq \delta\leq \delta_1$.
\end{proof}

\subsection{Local product structure}
We now deal with the existence of a local product structure. Though $\phi_t$ has no local product in the general case, $\phi_t$ has a related property. To look for strong stable and unstable sets we start with the horospherical leaves of the geodesic flow. These leaves enjoy a sort of weak local product structure provided by their heteroclinic connections (Theorem \ref{v1}): for every $\eta,\xi\in T_1\tilde{M}$ with $\xi\not\in \mathcal{\tilde{F}}^{cu}(-\eta)$, there exists $\theta\in T_1\tilde{M}$ such that
\begin{equation}\label{heti}
    \mathcal{\tilde{F}}^s(\eta)\cap \mathcal{\tilde{F}}^{cu}(\xi)=\mathcal{\tilde{I}}(\theta).
\end{equation}
Though these intersections always exist, they are generally not unique. This is because $\theta$ may be non-expansive and in such a case $\mathcal{\tilde{I}}(\theta)$ would be non-trivial.

The definition of $\tilde{\chi}$ strongly suggests that the quotients of the horospherical leaves are natural candidates to be the strong stable and unstable sets of $\psi_t$. For every $\eta\in T_1M$, we define
\[ V^s[\eta]=\chi(\mathcal{F}^s(\eta)), \qquad V^u[\eta]=\chi(\mathcal{F}^u(\eta)),  \]
\[  V^{cs}[\eta]=\chi(\mathcal{F}^{cs}(\eta)) \quad\text{ and }\quad V^{cu}[\eta]=\chi(\mathcal{F}^{cu}(\eta)). \]
We consider some connected components of $V^s[\eta]$ and $V^u[\eta]$. For every $[\eta]\in X$ and every open set $U\subset X$ containing $[\eta]$, we denote by $V^s[\eta]\cap U_c$ the connected component of $V^s[\eta]\cap U$ containing $[\eta]$. Similarly, we write $V^u[\eta]\cap U_c$, $V^{cs}[\eta]\cap U_c$ and $V^{cu}[\eta]\cap U_c$. Let $[\eta],[\xi]\in X$ close enough. If there exists an open set $U\subset X$ with $[\eta],[\xi]\in U$ such that $V^s[\eta]\cap U_c$ and $V^{cu}[\eta]\cap U_c$ intersect, then we define
\begin{equation}\label{intev}
V^s[\eta]\cap V^{cu}[\xi]=\left(V^s[\eta]\cap U_c\right)\cap \left(V^{cu}[\eta]\cap U_c\right).
\end{equation}
The heteroclinic connections \eqref{heti} provide $\theta\in T_1M$, $\tau\in \mathbb{R}$ and lifts $\tilde{\theta},\tilde{\eta},\tilde{\xi}\in T_1\tilde{M}$ of $\theta,\eta,\xi$ such that $\tilde{\xi}\not\in \mathcal{\tilde{F}}^{cu}(-\tilde{\eta})$ and 
\[    \mathcal{\tilde{F}}^s(\tilde{\eta})\cap \mathcal{\tilde{F}}^{cu}(\tilde{\xi})=\mathcal{\tilde{F}}^s(\tilde{\eta})\cap \mathcal{\tilde{F}}^{u}(\tilde{\phi}_{\tau}(\tilde{\xi}))=\mathcal{\tilde{I}}(\tilde{\theta}).  \]
\[   V^s[\eta]\cap V^{cu}[\xi]=V^s[\eta]\cap V^u(\psi_{\tau}[\xi])=\chi\circ d\pi(\mathcal{\tilde{I}}(\tilde{\theta}))=[\mathcal{I}(\theta)]=[\theta]. \]
By the definition of the quotient, this intersection is always unique if it exists. Denote by $B([\xi],r)$ the open ball of radius $r>0$ centered at $[\xi]$. 
\begin{lem}\label{basico}
Let $r_0>0$ be given by Lemma \ref{lev}. There cannot exist $\epsilon_0>0$ and sequences $[\eta_n],[\xi_n]\in X$, $t_n\in \mathbb{R}$ such that $t_n\to \infty$ and for every $n\geq 1$, $[\eta_n]\in V^s[\xi_n]$, $[\eta_n]$ belongs to the connected component of $V^s[\xi_n]\cap B([\xi_n],r_0)$ containing $[\xi_n]$,
\begin{equation}\label{sta}
    d([\eta_n],[\xi_n])\leq r_0 \quad\text{ and }\quad d(\psi_{t_n}[\eta_n],\psi_{t_n}[\xi_n])\geq \epsilon_0.
\end{equation}
An analogous statement holds for the unstable case.
\end{lem}
\begin{proof}
By contradiction suppose there exist the objects of the statement satisfying the inequalities \eqref{sta}. Lemma \ref{lev} says that there exist lifts $\tilde{\eta}_n,\tilde{\xi}_n\in T_1\tilde{M}$ of $\eta_n,\xi_n$ such that for every $n\geq 1$, $d_s(\tilde{\eta}_n,\tilde{\xi}_n)\leq Q+s_0$.

We claim that for every $n\geq 1$, $\tilde{\eta}_n\in \mathcal{\tilde{F}}^s(\tilde{\xi}_n)$. Otherwise, for some covering isometry $T$, $T(\tilde{\eta}_n)\in \mathcal{\tilde{F}}^s(\tilde{\xi}_n)$. Hence $[\eta_n]=[d\pi(T(\tilde{\eta}_n))]$ does not belong to the connected component of $V^s[\xi_n]\cap B([\xi_n],r_0)$ containing $[\xi_n]$ and the claim is proved. By lemma \ref{quesi} there exist $A,B>0$ such that for every $n\geq 1$,
\[   d_s(\phi_t(\tilde{\eta}_n),\phi_t(\tilde{\xi}_n))\leq Ad_s(\tilde{\eta}_n,\tilde{\xi}_n)+B\leq A(Q+s_0)+B=C, \text{ for every }t\geq 0.  \]
\begin{equation}\label{quasic}
\text{Hence}\quad d_s(\phi_t(\phi_{t_n}\tilde{\eta}_n),\phi_t(\phi_{t_n}\tilde{\xi}_n))\leq C \text{ for every }t\geq -t_n.
\end{equation}
Up to subsequences and using covering isometries we can assume that 
\begin{equation}\label{conv1}
\phi_{t_n}(\tilde{\eta}_n)\to \tilde{\eta} \quad\text{ and }\quad \phi_{t_n}(\tilde{\xi}_n)\to \tilde{\xi}.
\end{equation}
Since horospherical foliations are invariant by the action of the geodesic flow, we get $\phi_{t_n}(\tilde{\eta}_n)\in \mathcal{\tilde{F}}^s(\phi_{t_n}(\tilde{\xi}_n))$. Moreover, the continuity of the horospherical foliations gives $\tilde{\eta}\in \mathcal{\tilde{F}}^s(\tilde{\xi})$. Now, let $t\in \mathbb{R}$. Since $t_n \to \infty$, we see that $-t_n\leq t$ for $n$ large enough and inequalities \eqref{quasic} yields $d_s(\phi_t(\phi_{t_n}\tilde{\eta}_n),\phi_t(\phi_{t_n}\tilde{\xi}_n))\leq C$. By continuity we obtain $d_s(\phi_t(\tilde{\eta}),\phi_t(\tilde{\xi}))\leq C$ for every $t\in \mathbb{R}$. 

As $\tilde{\eta}\in \mathcal{\tilde{F}}^s(\tilde{\xi})$, Corollary \ref{key1} shows that $\tilde{\eta}\in \mathcal{\tilde{I}}(\tilde{\xi})$ hence $[\eta]=[\xi]$. Applying the map $\chi\circ d\pi$ to the sequences \eqref{conv1} we get
\[   \chi\circ d\pi (\phi_{t_n}(\tilde{\eta}_n))\to \chi\circ d\pi (\tilde{\eta}) \quad\text{ and }\quad \chi\circ d\pi (\phi_{t_n}(\tilde{\xi}_n))\to \chi\circ d\pi (\tilde{\xi}),   \]
\[ \psi_{t_n}[\eta_n]\to [\eta] \quad\text{ and }\quad \psi_{t_n}[\xi_n]\to [\xi]. \]
Thus $d(\psi_{t_n}[\eta_n],\psi_{t_n}[\xi_n])\to 0$ as $n\to \infty$. This contradicts inequalities \eqref{sta}. 
\end{proof}
An intermediate result to show the relationship between the pairs $W^{ss}[\eta]$,$W^{uu}[\eta]$ and $V^s[\eta]$,$V^u[\eta]$ is the so-called uniform contraction. We prove this contraction for $V^s[\eta]$ and $V^u[\eta]$ but only for distances smaller than $r_0$.
\begin{lem}\label{last}
Let $r_0>0$ be given by Lemma \ref{lev}. For every $\epsilon>0$ and $D\in(0,r_0]$ there exists $T>0$ such that if $[\eta]\in V^s[\xi]$, $[\eta]$ belongs to the connected component of $V^s[\xi]\cap B([\xi],r_0)$ containing $[\xi]$ and $d([\eta],[\xi])\leq D$, then 
\[ d(\psi_t[\eta],\psi_t[\xi])\leq \epsilon \quad\text{ for every } t\geq T. \]
An analogous result holds for the unstable case.
\end{lem}
\begin{proof}
By contradiction suppose there exist $\epsilon_0>0,D_0\in (0,r_0]$ and sequences $[\eta_n],[\xi_n]\in X$, $t_n\in \mathbb{R}$, such that $t_n\to \infty$ and for every $n\geq 1$, $[\eta_n]\in V^s[\xi_n]$, $[\eta_n]$ belongs to the connected component of $V^s[\xi_n]\cap B([\xi_n],r_0)$ containing $[\xi_n]$,
\[  d([\eta_n],[\xi_n])\leq D_0\leq r_0 \quad\text{ and }\quad d(\psi_{t_n}[\eta_n],\psi_{t_n}[\xi_n])\geq \epsilon_0.  \]
This contradicts Lemma \ref{basico} and proves the statement.
\end{proof}
As an immediate consequence we see that $V^s[\eta]$ and $V^u[\eta]$ agree with the strong sets of $\psi_t$ locally for distances smaller than $r_0$.
\begin{lem}\label{strong}
Let $r_0>0$ be given by Lemma \ref{lev}. If $[\eta]\in V^s[\xi]$, $[\eta]$ belongs to the connected component of $V^s[\xi]\cap B([\xi],r_0)$ containing $[\xi]$ and $d([\eta],[\xi])\leq r_0$ then
\[   d(\psi_t[\eta],\psi_t[\xi])\to 0 \quad\text{ as }\quad t\to \infty.  \]
In particular, $[\eta]\in W^{ss}[\xi]$. An analogous statement holds for the unstable case.
\end{lem}
\begin{proof}
For every $n\geq 1$, set $\epsilon_n=1/n$ and $D=r_0$ in Lemma \ref{last}. So, there exists a sequence $T_n\to \infty$ such that $d(\psi_t[\eta], \psi_t[\xi])\leq \frac{1}{n}$ for every $t\geq T_n$. This implies that $d(\psi_t[\eta], \psi_t[\xi])\to 0$ as $t\to \infty$.
\end{proof}
The local product requires not only the intersection of $W^{ss}[\eta]$ and $W^{uu}[\eta]$ but the intersection of the $\epsilon$-strong sets $W^{ss}_{\epsilon}[\eta]$ and $W^{uu}_{\epsilon}[\eta]$. The following lemma sets a criterion to identify points of $W^{ss}_{\epsilon}[\eta]$ and $W^{uu}_{\epsilon}[\eta]$. 
\begin{lem}\label{veci}
Let $r_0>0$ be given by Lemma \ref{lev}. For every $\epsilon>0$ there exists $\delta\in (0,r_0]$ such that if $[\eta]\in V^s[\xi]$, $[\eta]$ belongs to the connected component of $V^s[\xi]\cap B([\xi],r_0)$ containing $[\xi]$ and $d([\eta],[\xi])\leq \delta$, then
\[   d(\psi_t[\eta],\psi_t[\xi])\leq \epsilon \quad\text{ for every } t\geq 0. \]
An analogous result holds for the unstable case.
\end{lem}
\begin{proof}
By contradiction suppose there exist $\epsilon_0>0$ and sequences $[\eta_n],[\xi_n]\in X$, $\delta_n\in(0,r_0]$ and $t_n\in \mathbb{R}$, such that $\delta_n\to 0$ and for every $n\geq 1$, $[\eta_n]\in V^s[\xi_n]$, $[\eta_n]$ belongs to the connected component of $V^s[\xi_n]\cap B([\xi_n],r_0)$ containing $[\xi_n]$,  
\begin{equation}\label{sta2}
    d([\eta_n],[\xi_n])\leq \delta_n \leq r_0 \quad\text{ and }\quad d(\psi_{t_n}[\eta_n],\psi_{t_n}[\xi_n])\geq \epsilon_0.
\end{equation}
We claim that $t_n\to \infty$. Otherwise $t_n$ is bounded and after choosing a subsequence we have $t_n\to T\in \mathbb{R}$. For suitable subsequences, inequalities \eqref{sta2} imply that $[\eta_n]$ and $[\xi_n]$ converge to the same limit $[\eta]\in X$. Since $\psi_t$ is continuous, $\psi_{t_n}[\eta_n]$ and $\psi_{t_n}[\xi_n]$ converge to the same limit $\psi_T[\eta]$. This contradicts inequalities \eqref{sta2} and proves the claim. Since $t_n\to\infty$, inequalities \eqref{sta2} contradict Lemma \ref{basico} and prove the lemma.
\end{proof}
From this and Lemma \ref{strong}, we deduce that $W^{ss}_{\epsilon}[\eta]$ and $W^{uu}_{\epsilon}[\eta]$ agree with $V^s[\eta]$ and $V^u[\eta]$ locally. The following lemma states that if $[\eta]$ and $[\xi]$ are close enough then their intersection $[\theta]$ is close to $[\eta]$ and $[\xi]$.
\begin{lem}\label{inters}
For every $\epsilon>0$ there exists $\delta>0$ such that if $[\eta],[\xi]\in X$, $[\theta]\in V^s[\eta]\cap V^u(\psi_{\tau}[\xi])$ and $d([\eta],[\xi])\leq \delta$, then 
\[  d([\theta],[\eta])\leq \epsilon,\quad d([\theta],\psi_{\tau}[\xi])\leq \epsilon \quad\text{ and }\quad |\tau|\leq \epsilon. \]
\end{lem}
\begin{proof}
By contradiction suppose there exist $\epsilon_0>0$ and sequences $[\eta_n],[\xi_n],[\theta_n]\in X$, $\tau_n\in \mathbb{R}$ such that for every $n\geq 1$, $[\theta_n]\in V^s[\eta_n]\cap V^u(\psi_{\tau_n}[\xi_n])$, $|\tau_n|\geq \epsilon_0$,
\begin{equation}\label{sta3}
    d([\eta_n],[\xi_n])\leq \frac{1}{n}, \quad d([\theta_n],[\eta_n])\geq \epsilon_0 \quad \text{ and }\quad d([\theta_n],\psi_{\tau_n}[\xi_n])\geq \epsilon_0.
\end{equation}
Given $r_0>0$ from Lemma \ref{lev}, for every $n\geq 1$ large enough, $d([\eta_n],[\xi_n])\leq \frac{1}{n}\leq r_0$. So, we can choose lifts $\tilde{\eta}_n,\tilde{\xi}_n,\tilde{\theta}_n$ of $\eta_n,\xi_n, \theta_n$ such that $d_s(\tilde{\eta}_n,\tilde{\xi}_n)\leq Q+s_0$ and $\tilde{\theta}_n$ belongs to the fundamental domain containing $\tilde{\eta}_n$ and $\tilde{\xi}_n$. We claim that for every $n\geq 1$,
\begin{equation}\label{inte1}
\tilde{\theta}_n\in \mathcal{\tilde{F}}^s(\tilde{\eta}_n)\cap \mathcal{\tilde{F}}^u(\phi_{\tau_n}(\tilde{\xi}_n)).
\end{equation}
Otherwise there exist sequences of covering isometries $T_n,T'_n$ such that $\tilde{\theta}_n\in \mathcal{\tilde{F}}^s(T_n(\tilde{\eta}_n))\cap \mathcal{\tilde{F}}^u(\phi_{\tau_n}(T'_n(\tilde{\xi}_n)))$. Thus, there exists an open set $U\subset X$ containing $[\eta_n]$ such that $[\theta_n]$ does not belong to the connected component of $V^s[\eta_n]\cap U$ containing $[\eta_n]$. Similarly, $[\theta_n]$ does not belong to the connected component of $V^{cu}[\xi_n]\cap U$ containing $[\xi_n]$. This contradicts the definition of intersection and proves the claim.

So, if we use the same covering isometries for all sequences and choose suitable subsequences, we can assume that $\tilde{\eta}_n \to \tilde{\eta}, \tilde{\xi}_n\to \tilde{\xi}, \tilde{\theta}_n\to \tilde{\theta}$ and $\tau_n\to T$. Since we used the same covering isometries for the sequences, the continuity of the horospherical foliations applied to relation \eqref{inte1} yields
\begin{equation}\label{inte2}
\tilde{\theta}\in \mathcal{\tilde{F}}^s(\tilde{\eta})\cap \mathcal{\tilde{F}}^u(\phi_T(\tilde{\xi})). 
\end{equation}
We claim that $\tilde{\eta}\in \mathcal{\tilde{I}}(\tilde{\xi})$. Otherwise $\eta\not\in \mathcal{I}(\xi)$ and $d([\eta],[\xi])>0$. But applying the limit to inequalities \eqref{sta3}, we get $d([\eta],[\xi])=0$ and the claim is proved. 

The claim and relation \eqref{inte2} provide that $\tilde{\theta}\in \mathcal{\tilde{F}}^s(\tilde{\eta})\cap \mathcal{\tilde{F}}^u(\phi_T(\tilde{\eta}))$. From this, Corollary \ref{key1} shows that $\tilde{\theta} \in \mathcal{\tilde{I}}(\tilde{\eta})$. Therefore $[\theta_n]$ and $[\eta_n]$ converge to the same limit $[\theta]=[\eta]$, which contradicts inequalities \eqref{sta3}. This gives $\delta_1>0$ such that $d([\theta],[\eta])\leq \epsilon$. A similar reasoning yields $\delta_2>0$ such that $d([\theta],\psi_{\tau}[\xi])\leq \epsilon$. 

Finally, we see that $\tilde{\theta} \in \mathcal{\tilde{I}}(\tilde{\eta})\subset \mathcal{\tilde{F}}^u(\tilde{\eta})$ hence $\tilde{\theta} \in \mathcal{\tilde{F}}^u(\tilde{\eta})\cap \mathcal{\tilde{F}}^u(\phi_T(\tilde{\eta}))$. This holds if and only if $T=0$. Therefore $\tau_n \to 0$, contradicting $|\tau_n|\geq \epsilon_0$. We thus get $\delta_3>0$ such that $|\tau|\leq \epsilon$. Choosing $\delta=\min(\delta_1,\delta_2,\delta_3)$ we get the result.
\end{proof}
This result is important because it grants the continuity of the local product structure. 
\begin{lem}
The quotient flow $\psi_t$ has a local product structure.
\end{lem}
\begin{proof}
Let $r_0>0$ be given by Lemma \ref{lev} and $\epsilon\in(0,r_0]$. Consider $[\eta],[\xi]\in X$ and $[\theta]\in V^s[\eta]\cap V^u(\psi_{\tau}[\xi])$. By Lemma \ref{veci} there exists $\delta_1>0$ such that if $d([\theta],[\eta])\leq \delta_1$ and $d([\theta],\psi_{\tau}[\xi])\leq \delta_1$ then
\begin{equation}\label{veci1}
d(\psi_t[\theta],\psi_t[\eta])\leq \epsilon \quad\text{ and }\quad d(\psi_{-t}[\theta],\psi_{-t}\psi_{\tau}[\xi])\leq \epsilon, \quad \text{for} t\geq 0.
\end{equation}
For the same $\epsilon>0$, Lemma \ref{exp1} gives $\delta_2>0$ such that expansivity holds. Set $\delta_m=\min(\delta_1,\delta_2,\epsilon)$. By Lemma \ref{inters}, for $\delta_m>0$ there exists $\delta>0$ such that if $[\eta],[\xi]\in X$, $[\theta]\in V^s[\eta]\cap V^u(\psi_{\tau}[\xi])$ and $d([\eta],[\xi])\leq \delta$ then
\[ d([\theta],[\eta])\leq \delta_m,\quad d([\theta],\psi_{\tau}[\xi])\leq \delta_m \quad\text{ and }\quad |\tau|\leq \delta_m\leq \epsilon. \]
From this and $\delta_m\leq \epsilon\leq r_0$, Lemma \ref{strong} implies that $[\theta]\in W^{ss}[\eta]\cap W^{uu}(\psi_{\tau}[\xi])$. Furthermore, since $\delta_m\leq \delta_1$, $[\theta],[\eta],[\xi]$ satisfy inequalities \eqref{veci1} hence $[\theta]\in W^{ss}_{\epsilon}[\eta]\cap W^{uu}_{\epsilon}(\psi_{\tau}[\xi])$ and $|\tau|\leq \epsilon$.
\end{proof}
Finally, pseudo-orbit tracing and specification properties are a consequences of previous dynamical properties. More precisely, 
\begin{enumerate}
    \item By Theorem 7.1 of \cite{thom91}, if $\psi_t$ is expansive and has a local product structure then $\psi_t$ has the pseudo-orbit tracing property.
    \item By Proposition 6.2 of \cite{gelf19}, if $\psi_t$ is expansive, topological mixing and has the pseudo-orbit tracing property then $\psi_t$ has the specification property.
\end{enumerate}

\section{Uniqueness of the measure of maximal entropy of the geodesic flow}
This section is devoted to the study of the uniqueness of the measure of maximal entropy of the geodesic flow. The existence of such a measure follows from a work by Newhouse \cite{new89}. Indeed, the result says that a smooth flow on a compact smooth manifold always has a measure of maximal entropy. By hypothesis, the geodesic flow $\phi_t$ is a smooth flow acting on $T_1M$ and the result follows. We remark that in our case, the geodesic flow has positive topological entropy \cite{freire82}. 

The strategy for the proof of uniqueness of the measure of maximal entropy is the following. First of all, the properties of the factor flow implies that $\psi_t$ has a unique mesaure of maximal entropy. Secondly, we will show that the lift of this measure to $T_1M$ is the unique measure of maximal entropy for the geodesic flow. 

Recall that the quotient model is a quotient flow $\psi_t$ time-preserving semi-conjugate to the geodesic flow $\phi_t$. Bowen and Franco found a criterion to get the uniqueness of the measure of maximal entropy \cite{bowen71endo,fran77}.
\begin{teo}\label{frank}
Let $\phi_t:X\to X$ be a continuous flow acting on a compact metric space. If $\phi_t$ is expansive and has the specification property then $\phi_t$ has a unique measure of maximal entropy.
\end{teo}
From the previous section, Theorem \ref{propd} says that $\psi_t$ is expansive and has the specification property. Applying Theorem \ref{frank} to our case we see that $\psi$ has a unique measure of maximal entropy $\nu$.

To lift $\nu$ to $T_1M$ and verify the uniqueness property we rely on a abstract theorem proved by Buzzi-Fisher-Sambarino-Vasquez for discrete systems \cite{buzzi12}. They constructed a measure of maximal entropy using a classical argument due to Ledrappier and Walters \cite{ledra77}. We recall the construction for our setting. Let $\phi_t:Y\to Y$ and $\psi_t:X\to X$ be two continuous flows on compact metric spaces, $\chi:Y\to X$ be a time-preserving semi-conjugacy and $\nu$ be the measure of maximal entropy of $\psi_t$. Assume that $\psi_t$ is expansive, has the specification property and for every $x\in X$,
\begin{equation}\label{hip}
    h(\phi_1,\chi^{-1}(x))=0.
\end{equation}
Let $\epsilon>0$ be an expansivity constant for $\psi_t$. To prepare the ground for applying Buzzi-Fisher-Sambarino-Vasquez Theorem we let $Y=T_1M$, $\phi_t$ be the geodesic flow, $X$ be the quotient space, $\psi_t$ be the quotient flow and $\chi$ be the quotient map provided by Theorem \ref{main1}. For each $T>0$, we define the set
\[  Per(T,\epsilon)=\{ \chi^{-1}(\gamma)\subset Y: \gamma \text{ is a periodic orbit of }\psi_t \text{ with period in }[T-\epsilon,T+\epsilon] \}.\]
By expansivity this set is finite. The following lemma states a non-trivial fact about strips in our setting.  
\begin{lem}\label{notri}
Let $M$ be a compact surface without conjugate points of genus greater than one and $Per(T,\epsilon)$ be the set defined above. Then, every $\chi^{-1}(\gamma)\in Per(T,\epsilon)$ is a compact set invariant by the geodesic flow $\phi_t$. Moreover, if $S\in[T-\epsilon,T+\epsilon]$ is the period of $\gamma$, for any $\xi\in \chi^{-1}(\gamma)$ we have
\[  \chi^{-1}(\gamma)= \bigcup_{t\in [0,S]} \phi_t(\mathcal{I}(\xi)) = \phi_{[0,S]}(\mathcal{I}(\xi)). \]
In particular, its lift $\tilde{\phi}_{[0,S]}(\mathcal{\Tilde{I}}(\Tilde{\xi}))\subset T_1\Tilde{M}$ is a strip of bi-asymptotic orbits of the geodesic flow $\Tilde{\phi}_t$ for any lift $\tilde{\xi}\in T_1\Tilde{M}$ of $\xi$.
\end{lem}
\begin{proof}
Since $\gamma\subset X$ is a periodic orbit of $\psi_t$ we see that $\gamma$ is a closed set. The continuity of $\chi$ implies that $\chi^{-1}(\gamma)$ is closed in $T_1M$. Therefore since $T_1M$ is compact so is $\chi^{-1}(\gamma)$. Next, using Equation \eqref{semi} and the periodicity of $\gamma$, for every $t\in \mathbb{R}$ we have
\[  \phi_{-t}\circ \chi^{-1}(\gamma)=(\phi_t)^{-1}\circ \chi^{-1}(\gamma)= \chi^{-1}\circ (\psi_t)^{-1}(\gamma)= \chi^{-1}\circ \psi_{-t}(\gamma) = \chi^{-1}(\gamma).  \]
Therefore $\chi^{-1}(\gamma)$ is invariant by $\phi_t$. Finally, for any $\xi\in \chi^{-1}(\gamma)$ we can write $\gamma=\bigcup_{t\in [-S,0]}\psi_{-t}(\chi(\xi))$. Using Equation \eqref{semi} again we get
\begin{align*}
\chi^{-1}(\gamma) &= \chi^{-1} \left( \bigcup_{t\in [-S,0]}\psi_{-t}(\chi(\xi)) \right) = \bigcup_{t\in [-S,0]}\chi^{-1}\circ \psi_{-t}\circ \chi(\xi) \\
   &= \bigcup_{t\in [-S,0]}\phi_{-t}\circ \chi^{-1} \circ \chi(\xi) =\bigcup_{t\in [0,S]}\phi_{t}(\mathcal{I}(\xi)) =\phi_{[0,S]}(\mathcal{I}(\xi)).
\end{align*}
\end{proof}
Observe that applying Brouwer fixed point Theorem to the restricted map $\phi_S:\mathcal{I}(\xi)\to \mathcal{I}(\xi)$ implies that $\chi^{-1}(\gamma)$ contains at least one periodic orbit (with period $S$). For the other remaining orbits in $\chi^{-1}(\gamma)$ we do not know a priori if they are actually periodic. However, for every non-periodic $\phi_t$-orbit $\beta\subset\chi^{-1}(\gamma)$, Lemma \ref{notri} implies that $\beta$ and its accumulation points remain in $\chi^{-1}(\gamma)$. These properties might help to understand the geometry of these particular strips in future studies. Note that for compact surfaces without focal points, the geometry of strips is well-understood due to the flat strip Theorem \cite{pesin77}.

It also follows from Lemma \ref{notri} that there exists a probability measure $\mu_{\gamma}$ supported on $\chi^{-1}(\gamma)$ and invariant by the geodesic flow $\phi_t$. So, we can take the average 
\[  \mu_T=\frac{\sum_{\gamma}\mu_{\gamma}}{\# Per(T,\epsilon)},  \]
where $\gamma$ varies according to $\chi^{-1}(\gamma)\in Per(T,\epsilon)$. We see that $\mu_T$ is a probability measure on $Y$ invariant by the flow $\phi_t$. Let $\mu\in \mathcal{M}(\phi)$ be an accumulation point of the set $(\mu_T)_{T>0}$ in the weak$^*$ topology. So, there exists a sequence $T_n\to \infty$ such that $\mu_{T_n}\to \mu$ weakly.

Notice that for every $\chi^{-1}(\gamma)\in Per(T,\epsilon)$, $\chi_*\mu_{\gamma}$ is a probability measure supported on $\gamma$ and invariant by the flow $\psi_t$. It follows that $\chi_*(\mu_{T_n})$ is a probability measure supported on the union of periodic orbits of $\psi_t$ with period in $[T-\epsilon,T+\epsilon]$. In this case, Bowen showed that $\chi_*(\mu_{T_n})\to \nu$ in the weak$^*$ topology \cite{bowenper}. The continuity of $\chi_*$ and $\mu_{T_n}\to \mu$ provide that $\chi_*\mu=\nu$.

We verify that $\mu$ is a measure of maximal entropy. Since $(Y,\phi_t,\mu)$ is an extension of $(X,\psi_t,\nu)$, we have $h_{\nu}(\psi_1)\leq h_{\mu}(\phi_1)$. Applying Bowen's formula \cite{bowen71endo} and assumption \eqref{hip}, we conclude that 
\[  h(\phi_1)\leq h(\psi_1)+\sup_{x\in X}h(\phi_1,\chi^{-1}(x))=h(\psi_1)\quad \text{ hence }\quad h(\phi_1)= h(\psi_1).   \]
Since $\nu$ is a measure of maximal entropy for $\psi_t$, $h(\phi_1)=h(\psi_1)=h_{\nu}(\psi_1)\leq h_{\mu}(\phi_1)$. So, every accumulation point $\mu\in \mathcal{M}(\phi)$ of the set $(\mu_T)_{T>0}$ satisfies:
\begin{equation}\label{par}
    \mu \text{ is a measure of maximal entropy for } \phi_t \quad \text{ and } \quad  \chi_*\mu=\nu.
\end{equation}
We state Buzzi-Fisher-Sambarino-Vasquez's Theorem for continuous systems. The proof is analogous to the discrete case with minor changes.
\begin{pro}\label{samba}
Let $\phi_t:Y\to Y$ and $\psi_t:X\to X$ be two continuous flows on compact metric spaces, $\chi:Y\to X$ be a time-preserving semi-conjugacy and $\nu$ be the measure of maximal entropy of $\psi_t$. Assume that $\psi_t$ is expansive, has the specification property and  
\begin{enumerate}
\item $h(\phi_1,\chi^{-1}(x))=0$ for every $x\in X$.
\item $\nu \bigg(\{ \chi(y): \chi^{-1}\circ\chi(y)=\{y\} \}\bigg)=1$.
\end{enumerate}
Then, there exists a unique measure of maximal entropy $\mu$ of $\phi_t$ with $\chi_*\mu=\nu$.
\end{pro}
We apply this proposition to our context. Let $Y=T_1M$, $\phi_t$ be the geodesic flow, $X$ be the quotient space, $\psi_t$ be the quotient flow, $\chi$ be the quotient map and $\nu$ be the unique measure of maximal entropy of $\psi_t$. With these choices, the assumptions of Proposition \ref{samba} are satisfied except for Hypothesis 1 and 2. Regarding Hypothesis 1, we see that for every $[\eta]\in X$, 
\begin{equation}\label{ayu}
    \chi^{-1}[\eta]=\chi^{-1}\circ \chi(\eta)=\mathcal{I}(\eta).
\end{equation}
For compact surfaces without conjugate points and genus greater than one, Gelfert and Ruggiero \cite{gelf20} proved that $h(\phi_1, \mathcal{I}(\eta))=0$ for every $\eta\in T_1M$. Therefore, Hypothesis 1 is satisfied. Moreover, condition \eqref{par} says that there exists a measure of maximal entropy $\mu$ for the geodesic flow $\phi_t$ such that $\chi_*\mu=\nu$. So, to show the uniqueness it only remains to prove Hypothesis 2.

We express Hypothesis 2 of Proposition \ref{samba} in our context. By identity \eqref{ayu}, this hypothesis has the following form
\[ \{ \chi(y): \chi^{-1}\circ \chi(y)=\{y\} \}= \{ \chi(\eta)\in X: \mathcal{I}(\eta)=\{\eta\} \}=\chi(\mathcal{R}_0). \]
Consequently, Hypothesis 2 becomes
\begin{equation}\label{ayu1}
\nu(\chi(\mathcal{R}_0))=1.
\end{equation}
To prove this condition, we use Proposition 3.3 of Climenhaga-Knieper-War's work \cite{clim21}. This proposition states a classical Katok's result in the context of geodesic flows of surfaces. 
\begin{lem}\label{ani}
Let $M$ be a surface without conjugate points of genus greater than one and $\mu$ be an ergodic measure on $T_1M$ invariant by the geodesic flow. 
\[  \text{If }\quad h_{\mu}(\phi_1)>0 \quad \text{ then }\quad \mu(\mathcal{R}_0)=1. \]
\end{lem}
We prove below condition \eqref{ayu1} and so the uniqueness of the measure of maximal entropy for the geodesic flow $\phi_t$.
\begin{proof}
As remarked above, by condition \eqref{par}, $\mu$ is a measure of maximal entropy and hence $h_{\mu}(\phi_1)=h(\phi_1)>0$. Ergodic decomposition of $\mu$ provides an ergodic component $\tau$ with $h_{\tau}(\phi_1)>0$. Lemma \ref{ani} implies that $\tau(\mathcal{R}_0)=1$ hence $\mu(\mathcal{R}_0)>0$. So, we have 
\[  \nu(\chi(\mathcal{R}_0))=\chi_*\mu(\chi(\mathcal{R}_0))=\mu(\chi^{-1}\chi\mathcal{R}_0)=\mu(\mathcal{R}_0)>0. \]
Since $\nu$ is ergodic and $\chi(\mathcal{R}_0)$ is invariant by $\psi_t$, we get $\nu(\chi(\mathcal{R}_0))=1$.
\end{proof}
Finally, we remark that Climenhaga-Knieper-War \cite{clim21} also showed that the unique measure of maximal entropy has full support. This property can be proven by our methods assuming that the expansive set $\mathcal{R}_0$ is dense. For this, we first restate Proposition 7.3.15 of \cite{fish19} in our context.
\begin{pro}\label{gibbs}
Let $X$ be a compact metric space, $\psi_t:X\to X$ be a continuous expansive flow with the specification property and $\nu$ be its unique measure of maximal entropy. Then, for every $\epsilon>0$ there exist $A_{\epsilon},B_{\epsilon}>0$ such that for every $x\in X$ and every $T>0$, we have $A_{\epsilon}\leq e^{Th(\psi_1)}\nu(B(x,\epsilon,T))\leq B_{\epsilon}$ where $B(x,\epsilon,T)$ is the $(T,\epsilon)$-dynamical ball defined in Equation \eqref{dyna} in Subsection \ref{p5}. 
\end{pro}
\begin{pro}
Let $M$ be a compact surface without conjugate points of genus greater than one and $\mu$ be its unique measure of maximal entropy. If $\mathcal{R}_0$ is dense in $T_1M$ then $\mu$ has full support.
\end{pro}
\begin{proof}
For $T=0$ and every $\epsilon>0$, apply Proposition \ref{gibbs} to the quotient flow $\psi_t:X\to X$ and its unique measure of maximal entropy $\nu$. So, we have $0<A_{\epsilon}\leq \nu(B(x,\epsilon,0))\leq B_{\epsilon}$ for every $\epsilon>0$ and every $x\in X$. Therefore $\nu$ has full support on $X$ since $B(x,\epsilon,0)$ is just an open ball of radius $\epsilon$ centered at $x$. Now, let $U$ be any open set of $T_1M$. By density of $\mathcal{R}_0$, there is an expansive point $\xi\in U$. Note that the family of open saturated neighborhoods around $\mathcal{I}(\xi)=\xi$ defined in Section \ref{s3} forms a basis of neighborhoods at $\xi$. Hence there exists an open saturated set $A=A(\xi,\epsilon', \delta',\tau')$ included in $U$. Since $\nu$ has full support and $\chi(A)$ is an open set of $X$, the conclusion follows from
\[  \mu(U)\geq \mu(A)=\mu(\chi^{-1}\chi(A))=\chi_*\mu(\chi(A))=\nu(\chi(A))>0.  \]
\end{proof}
Despite we assumed that $\mathcal{R}_0$ is dense, we believe that this actually holds in our setting. For the moment, this density for the more general case is being studied for future work. This property holds for example for compact higher genus surfaces without conjugate points and with continuous Green bundles which include the case of surfaces without focal points.

\section{Acknowledgments}
I would like to thank my advisor Rafael Ruggiero for useful discussions. I appreciate the financial support of CAPES and FAPERJ funding agencies during the work. This article was supported in part by INCTMat under the project INCTMat-Faperj (E26/200.866/2018).

\bibliographystyle{plain}

\end{document}